\documentclass[a4paper,11pt]{amsart}

\usepackage{epsfig}
\usepackage{amsthm,amsfonts}
\usepackage{amssymb,graphicx,color}
\usepackage[all]{xy}
\usepackage{verbatim}
\usepackage{hyperref}
\usepackage{enumitem}

\usepackage[normalem]{ulem}
\usepackage{xcolor}

\newtheorem{theorem}{Theorem}[section]
\newtheorem*{theorem*}{Theorem}
\newtheorem{lemma}[theorem]{Lemma}
\newtheorem{corollary}[theorem]{Corollary}
\newtheorem{proposition}[theorem]{Proposition}
\newtheorem{definition}[theorem]{Definition}

\newtheorem{example}[theorem]{Example}

\newtheorem{innercustomthm}{Theorem}
\newenvironment{customthm}[1]
  {\renewcommand\theinnercustomthm{#1}\innercustomthm}
  {\endinnercustomthm}
\newcommand{\R}{\mathbb{R}}
\newcommand{\C}{\mathbb{C}}

\begin{document}

\title[On the Invariance of the Real Milnor Number]
{On the Invariance of the Real Milnor Number under Asymptotically Lipschitz Equivalence}

 \author[Raphael de Omena]{Raphael de Omena}
\address{Raphael de Omena: Departamento de Matemática, Universidade Federal do Ceará, 
Campus do Pici, Bloco 914, Pici, 60440-900, Fortaleza--CE, Brasil. 
\newline E-mail: \texttt{raphael.marinho@mat.ufc.br}}

\author[José Edson Sampaio]{José Edson Sampaio}
\address{José Edson Sampaio: Departamento de Matemática, Universidade Federal do Ceará, 
Campus do Pici, Bloco 914, Pici, 60440-900, Fortaleza--CE, Brasil. 
\newline E-mail: \texttt{edsonsampaio@mat.ufc.br}}

\author[Emanoel Souza]{Emanoel Souza}
\address{Emanoel Souza: Departamento de Matemática, Universidade Estadual do Ceará (UECE), 
Campus Itaperi, Av. Dr. Silas Munguba, 60714-903, Fortaleza--CE, Brasil. 
\newline E-mail: \texttt{emanoelfs.cdd@gmail.com}}

%
%\subjclass[2010]{14B05, 32S50, 58K30 (Primary) 58K20 (Secondary)}

\begin{abstract}  
We investigate sufficient conditions for the invariance of the real Milnor number under \( \mathcal{R} \)-bi-Lipschitz equivalence for function-germs \( f, g \colon (\mathbb{R}^n, 0) \to (\mathbb{R}, 0) \). More generally, we explore its invariance within the extended framework of \( \mathcal{R} \)-asymptotically Lipschitz equivalence. To this end, we introduce the \(\alpha\)-derivative, which provides a natural setting for studying asymptotic growth. Additionally, we discuss the implications of our results in the context of \( C^k \) and \( C^{\infty} \) equivalences, establishing sufficient conditions for the real Milnor number to remain invariant.  
\end{abstract}  

\maketitle

\section{Introduction}

An important approach in singularity theory involves studying the properties of map germs under various equivalence relations. These equivalence relations are defined by Mather's groups, for instance \(\mathcal{G} \in \{\mathcal{A}, \mathcal{R}, \mathcal{L}, \mathcal{K}, \mathcal{C}\}\), which act on the space of map germs.

Since listing all $\mathcal{G}$-orbits is impractical, certain restrictions must be imposed. In general, we focus on $\mathcal{G}$-finite germs of lower $\mathcal{G}$- or $\mathcal{G}_e$-codimension. These codimensions are $\mathcal{G}$-invariants of the function-germ $f$, meaning they are values that depend exclusively on the $\mathcal{G}$-orbit of $f$.

Let \(f \colon (\mathbb{R}^n,0) \to \mathbb{R} \) be a germ of $C^{\infty}$ function. The \textit{codimension} and the \textit{extended codimension} of the $\mathcal{G}$-orbit of \( f \) is defined, respectively, by
\[
\text{cod}(f, \mathcal{G}) = \dim_{\mathbb{R}}\dfrac{\mathfrak{m}_n\mathcal{E}_n}{L\mathcal{G} \cdot f} \quad \text{and} \quad \text{cod}_e(f, \mathcal{G}) = \dim_{\mathbb{R}}\dfrac{\mathcal{E}_n}{L_e\mathcal{G} \cdot f},
\]
where \( L\mathcal{G} \cdot f \) and \( L_e\mathcal{G} \cdot f \) denote the tangent space and the extended tangent space of the orbit \( \mathcal{G} \cdot f \). 

Among the various invariants associated with function-germs, the extended codimension in the 
\(\mathcal{R}\) equivalence plays a fundamental role. Notably, the \textit{real Milnor number} is defined as \(\mu(f) = \text{cod}_e(f, \mathcal{R})\). In Section \ref{section:preliminaries}, we present the necessary background on Lipschitz equivalence and the real Milnor number. Here, we briefly highlight the main problem addressed in this work.

In the context of holomorphic function-germs \( f \colon (\mathbb{C}^n, 0) \to (\mathbb{C}, 0) \), Milnor \cite{Milnor:1968} showed that if \( f \) has an isolated singularity at the origin, then the intersection of the singular fiber with a sufficiently small sphere is homotopy equivalent to a wedge of \( \mu(f) = \deg\left(\frac{\nabla f}{|\nabla f|}\right) \) spheres.

Lê \cite{Le:1973} proved that the Milnor number is a topological invariant for germs of holomorphic functions $ f, g \colon (\mathbb{C}^n, 0) \to (\mathbb{C}, 0) $ with isolated singularities at the origin. Specifically, if there exists an ambient homeomorphism \linebreak $ \varphi \colon (\mathbb{C}^n, 0) \to (\mathbb{C}^n, 0) $ such that $ \varphi(V(f)) = V(g) $, then $ \mu(f) = \mu(g) $.

However, the same does not hold for real $C^{\infty}$ functions, as shown by Wall \cite{Wall:1983} in the following example. In the same work, Wall also proved that the real Milnor number modulo 2 is a topological invariant.

\begin{example}[\cite{Wall:1983}] 
Let $f_{ij} \colon \mathbb{R}^n \to \mathbb{R}$ be the polynomial given by
$$
f_{ij}(x_1, \dots, x_n) = x_1^{2i} + x_2^2 + \cdots + x_j^2 - x_{j+1}^2 - \cdots - x_n^2.
$$
Then, the topological type of $V(f_{ij})$ is determined by $j$ and $n$, while $\mu(f_{ij}) = 2i - 1$.
\end{example}

As demonstrated by the following example, the real Milnor number is not a $ C^1 $ invariant. In particular, the real Milnor number is not a bi-Lipschitz invariant.

\begin{example}\label{Ex:NotC1}
Let $ f, g \colon \mathbb{R}^2 \to \mathbb{R} $ be polynomials given by 
 \[f(x,y) = x^4 - y^3, \quad g(x,y) = y. \] 
 The map $ \phi \colon \mathbb{R}^2 \to \mathbb{R}^2 $ defined by $ \phi(x,y) = \left(x, x^{\frac{4}{3}} - y \right) $ is a $ C^1 $ diffeomorphism such that $ \phi(V(f)) = V(g) $. Nevertheless, $ \mu(f) = 6 $ and $ \mu(g) = 0 $.
\end{example}

In Section~\ref{section:invarianceMilnor}, we analyze conditions under which the real Milnor number remains invariant and establish that:

\begin{customthm}{\ref*{Cor:Invariance}}
   Let \(f, g\colon (\R^n,0)\to (\R,0)\) be germs of \(C^{\infty}\) functions. Suppose that \({\rm in}(f)\) and \({\rm in}(g)\) have algebraically isolated singularities at the origin. If \(f\) and \(g\) are \(\mathcal{R}\)-asymptotically Lipschitz equivalent  at the origin, then \(\mu(f) = \mu(g)\).
\end{customthm}

Examples~\ref{Ex:AlphaHolder} and~\ref{exam:Lip_trivial_milnor_num_different} highlight the sharpness of Theorem \ref{Cor:Invariance} in the following sense. Example~\ref{Ex:AlphaHolder} shows that the regularity assumption on the equivalence is optimal: the real Milnor number is not preserved under a bi-$\alpha$-Hölder homeomorphism when $0 < \alpha < 1$ even under the assumptions that \({\rm in}(f)\) and \({\rm in}(g)\) have algebraically isolated singularities at the origin. On the other hand, Example~\ref{exam:Lip_trivial_milnor_num_different} reveals that it is crucial to assume that the initial parts of the functions have algebraically isolated singularities.

To establish this result, we introduce the 
\(\alpha\)-derivative in Section \ref{section:asymptotic}, which naturally extends the notion of the initial part and behaves well under asymptotic Lipschitz equivalence. Moreover, we generalize results for mappings and functions within a framework based on asymptotic growth. To be specific, we demonstrate the invariance of the order of a map under this equivalence (Theorem \ref{inv_order}).

Finally, in Section \ref{section:diffeo}, we extend our discussion to the setting of \(C^k\) equivalence, showing that, \textit{a priori}, it is not necessary for both initial parts to have an algebraically isolated singularity at the origin (Theorem \ref{Teo:CkEqui}). Furthermore, in the case of \(C^{\infty}\) equivalence, we establish the following result:
\begin{customthm}{\ref*{Teo:Cinfinito}}
Let \( f, g \colon (\mathbb{R}^n,0) \to (\mathbb{R},0) \) be irreducible germs of real analytic functions. If there exists a \( C^{\infty} \) diffeomorphism  
\(\phi \colon (V(f), 0) \to (V(g), 0)\), then  \( \mu(f) = \mu(g) \).
\end{customthm}

\section{Preliminaries}\label{section:preliminaries}

Here, we explore the basic definitions relevant to this paper, specifically those related to \textit{lipeomorphisms} \( \phi \colon \mathbb{R}^n \to \mathbb{R}^n \), \textit{i.e}, a bijective map for which there exists a constant \( c > 0 \) such that
\[
\frac{1}{c} \|x - y\| \leq \|\phi(x) - \phi(y)\| \leq c \|x - y\| \quad \text{for all } x, y \in \mathbb{R}^n.
\]

An important object related to a lipeomorphism is constructed in \cite{Sampaio:2016} as follows. Let \( \varphi: U \rightarrow V \) be a lipeomorphism, where \( U, V \subset \mathbb{R}^n \) are open subsets. Assume that \( \phi, \psi: \mathbb{R}^n \rightarrow \mathbb{R}^n \) are Lipschitz extensions of \( \varphi \) and \( \varphi^{-1} \), respectively. For a fixed point \( x \in U \), define the sequence of functions \( \phi_{x,j}: \mathbb{R}^n \rightarrow \mathbb{R}^n \) by
\[
\phi_{x,j}(v) = j\left(\phi\left(x + \frac{1}{j}v\right) - \phi(x)\right).
\]

Since \( \phi \) is Lipschitz, there exists a map \( d_x\phi: \mathbb{R}^n \rightarrow \mathbb{R}^n \) and a subsequence \( \{j_i\} \) such that $\phi_{x,j_i}$ converges, uniformly on compact subsets containing $x$, to $d_x\phi$, by Arzel\`a-Ascoli Theorem. Setting \( y = \phi(x) \), we can obtain a map \( d_y\psi \) by applying the same construction to \( \psi \) at \( y \). It is worth noting that \( d_x\phi \) and \( d_y\psi \) are lipeomorphisms with \( (d_x\phi)^{-1} = d_y\psi \) on the open subsets.

\begin{definition}
    A map \( d_x\phi: \mathbb{R}^n \rightarrow \mathbb{R}^n \), constructed as in the preceding discussion, is called a \textbf{pseudo-Lipschitz derivative} of \( \phi \) at \( x \in U \).
\end{definition}

\(\mathcal{G}\)-bi-Lipschitz equivalence has been studied by several authors, including \cite{BirbrairCFR:2007,HenryP:2004, NguyenRT:2020}, who investigated classification and finiteness results. To proceed, we first recall the definition that is used throughout this article.

\begin{definition}
Two germs of mappings \( F, G : (\mathbb{R}^n, 0) \to (\mathbb{R}^p, 0) \) are said to be \(\mathcal{R}\)-\textbf{bi-Lipschitz equivalent} if there exist a germ of lipeomorphism 
\( \phi : (\mathbb{R}^n, 0) \to (\mathbb{R}^n, 0) \) such that \(F = G \circ \phi\).
\end{definition}

%For germs of analytic functions \( f \colon (\mathbb{R}^n, 0) \to (\mathbb{R}, 0) \), the authors in \cite{HenryP:2004} construct an invariant for \(\mathcal{R}\)-bi-Lipschitz equivalence that varies continuously in many analytic families. This confirms that the \(\mathcal{R}\)-bi-Lipschitz equivalence of analytic function germs admits continuous moduli. On the other hand, the authors in \cite{BirbrairCFR:2007} consider the problem of \(\mathcal{K}\)-bi-Lipschitz classification of polynomial function germs. They show that this classification problem is tame, meaning that it has no moduli.

\begin{definition} 
Two mappings \( F, G \colon \mathbb{R}^n \to \mathbb{R}^p \) are said to be \textbf{asymptotically equivalent} at a point \(x_0 \in \mathbb{R}^n\) if there exist a neighborhood \(U_{x_0}\) of \(x_0\) in \(\mathbb{R}^n\) and positive constants \(c_1, c_2\) such that  
\[
c_2 \|F(x)\| \leq \|G(x)\| \leq c_1 \|F(x)\|, \quad \text{for all } x \in U_{x_0}.
\]  
In this case, we write \(F \sim_{x_0} G\). If this condition holds for all \(x\) in the domain of \(F\) and \(G\), we simply denote \(F \sim G\).  

Furthermore, \(F\) and \(G\) are said to be \textbf{asymptotically Lipschitz equivalent} at \(x_0\) if there exist a lipeomorphism \(\phi \colon U \to V\), where \(U \subset \mathbb{R}^n\) and \(V \subset \mathbb{R}^p\) are open subsets, and a constant \(c \geq 1\) such that  
\[
\frac{1}{c} \|F(x) - F(x_0)\| \leq \|G(\phi(x)) - G(\phi(x_0))\| \leq c \|F(x) - F(x_0)\|, 
\]  
for all \(x \in U_{x_0}\). In this case, we denote \(F \approx_{x_0} G\). If this property holds for all \(x\) in \(U\), we simply write \(F \approx G\).  

Finally, \(F\) and \(G\) are called \(\mathcal{R}\)-\textbf{asymptotically Lipschitz equivalent} at \(x_0\) if there exists a lipeomorphism \(\phi\) such that \(F \sim_{x_0} G \circ \phi\).  
\end{definition}

The authors in \cite{BirbrairCFR:2007} proved that if \(f\) and \(g\) are germs of Lipschitz functions, then they are asymptotically equivalent if and only if they are \(\mathcal{C}\)-bi-Lipschitz equivalent. Furthermore, it has been demonstrated in \cite{Nguyen:2022} that for \(C^\infty\) mapping germs, \(\mathcal{K}\)-bi-Lipschitz equivalence implies \(\mathcal{R}\)-asymptotically Lipschitz equivalence.

%\subsection{Milnor Number of Real Functions}

We now consider $\mathcal{E}_n$, the set of real $C^{\infty}$ function-germs $(\mathbb{R}^n, 0) \to \mathbb{R}$. It is well known that $\mathcal{E}_n$ is naturally endowed with an $\mathbb{R}$-algebra structure, and the set $\mathfrak{m}_n = \{ f \in \mathcal{E}_n \mid f(0) = 0 \}$ constitutes the unique maximal ideal in $\mathcal{E}_n$.

In the context of the previously defined $\mathcal{R}$-bi-Lipschitz equivalence, if the lipeomorphism is a \(C^k\) diffeomorphism, for \(k \in \mathbb{N} \cup \{\infty\}\) we refer to these as $\mathcal{R}$-\(C^k\) \textit{equivalence}.

One should observe that two function-germs \(f,g \in \mathcal{E}_n \) are $\mathcal{R}$-\(C^{\infty}\) equivalent if they lie in the same orbit under the action of the group $\mathcal{R}$ on $\mathcal{E}_n$. In this regard, many authors have investigated the notions of stability, deformation, and unfolding. For precise definitions, you may refer to Chapter 3 of \cite{IzumiyaFRT:2016}. 

\begin{definition}
    Let \( f \in \mathcal{E}_n \). We define the real \textbf{Milnor number} of \( f \) as the length of the algebra
    \[
    \mu(f) = \text{cod}_e(f, \mathcal{R}) := \dim_{\mathbb{R}} \frac{\mathcal{E}_n}{J(f)},
    \]
    where \( J(f) \) is the ideal generated by \( \left\{ \frac{\partial f}{\partial x_1}, \dots, \frac{\partial f}{\partial x_n} \right\} \).

Additionally, we say that $ f $ has an \textbf{algebraically isolated singularity} at the origin if $ \mu(f) < +\infty $.

\end{definition}

The real case is connected with the complex case, as in the following result, which can be found as Proposition 7.1.2 in \cite{NunoBallesterosOSR:2021}.

\begin{proposition}\label{Prop:Complexifing} 
Let \( f \colon (\mathbb{R}^n, 0) \to (\mathbb{R}, 0) \) be an analytic function-germ, and let \( f_{\mathbb{C}} \) be its complexification. Then, \( \mu(f) < \infty \) if and only if \( \mu(f_{\mathbb{C}}) < \infty \). In addition, if \( \mu(f) < \infty \) then \( \mu(f) = \mu(f_{\mathbb{C}}) \).   
\end{proposition}

\section{Asymptotic Lipschitz Invariants of Real Mappings} \label{section:asymptotic}

In this section, we investigate the behavior of real mappings by introducing the concept of asymptotic Lipschitz invariants. We define these invariants and examine their fundamental properties, with a focus on their behavior under asymptotic bi-Lipschitz transformations. These results extend the framework of \(\mathcal{G}\)-bi-Lipschitz equivalence.

\begin{definition}
Let \( F: U \subset \mathbb{R}^n \to \mathbb{R}^p \) be a map, where \( U \) is an open subset. We define the \textbf{order} of \( F \) at a point \( x \in U \) to be a real number \( \alpha \geq 0 \) if the limit
\[
L = \lim_{t \to 0^+} \frac{F(x + tv)}{t^{\alpha}} 
\]
exists for every nonzero vector \( v \in \mathbb{R}^n \setminus \{0\} \), and for at least one such \( v \), the limit satisfies \( L \in \mathbb{R}^p \setminus \{0\} \). In this case, we write \( \operatorname{ord}_x F = \alpha \).

If, for every \( m \in \mathbb{N} \) and some \( v \in \mathbb{R}^n \), we have
\[
\lim_{t \to 0^+} \frac{F(x + tv)}{t^m} = 0,
\]
then we define \( \operatorname{ord}_x F = +\infty \).

Moreover, if \( F(x) \neq 0 \), we define \( \operatorname{ord}_x F = 0 \).
\end{definition}

\begin{definition}
A map $F: U \subset \mathbb{R}^n \to \mathbb{R}^p$ is said to have an $\alpha$\textbf{-directional derivative at $x \in U$ in the direction} $v \in \mathbb{R}^n$ if the following limit exists:
\[
\frac{\partial^{\alpha} F}{\partial_{+} v^{\alpha}}(x) = 
\lim_{t \to 0^+} \frac{F(x + tv) - F(x)}{t^{\alpha}}.
\]

Furthermore, if there exists a continuous map $H: \mathbb{R}^n \to \mathbb{R}^p$ such that
\[
F(x + tv) - F(x) = t^{\alpha} H(v) + o_{\alpha}(tv), \forall v \in \mathbb{R}^n
\]
where $\lim_{t \to 0^+} \frac{o_{\alpha}(tv)}{t^{\alpha}} = 0$, we denote $H(v) = D^{\alpha}_{+} F_x(v)$ and refer to it as the $\alpha$-\textbf{derivative} of $F$ at $x$.
\end{definition}

\begin{example} \label{Ex:HalfOrd}
   Let $f: \mathbb{R} \rightarrow \mathbb{R}$ be a function defined by
\[
f(x) = \left\{
\begin{array}{ll} 
\frac{x}{|x|^{1/2}}, & \text{if } x \neq 0,   \\ 
0, & \text{if } x = 0.
\end{array}
\right.
\]

Then, ${\rm ord}_0 f= \frac{1}{2}$. Moreover,
\[
D^{\alpha}_{+} f_x(v) = \left\{
\begin{array}{ll} 
0, & \text{if } \alpha < \frac{1}{2},   \\
\frac{v}{|v|^{1/2}}, & \text{if } \alpha = \frac{1}{2},   \\ 
\text{does not exist}, & \text{if } \alpha > \frac{1}{2}.
\end{array}
\right.
\]
\end{example}

Notice that the existence of the order $\alpha$ of a map does not necessarily imply the existence of its $\alpha$-derivative. Consider the following example.

\begin{example} \label{Ex:OrdZero}
    Let \( f \colon \mathbb{R}^2 \to \mathbb{R} \) be defined by  
    \[
    f(x, y) = 
    \begin{cases} 
    \frac{xy}{x^2 + y^2}, & \text{if } (x, y) \neq (0, 0), \\ 
    0, & \text{if } (x, y) = (0, 0).
    \end{cases}
    \]

    Observe that $\alpha = {\rm ord}_{(0,0)}f = 0$ and the candidate for \( D_{+}^{\alpha}f_{(0,0)}(v) \) should be \( f(v) \), since \( f(tv) = f(v) \). However, \( f \) is not a continuous function.
\end{example}

The following proposition establishes properties of the \(\alpha\)-derivative.

\begin{proposition}\label{Prop:Properties}
    Let \( F: U \subset \mathbb{R}^n \to  \mathbb{R}^p \) be a map that admits the \(\alpha\)-derivative at \( x \). The following properties hold:
    \begin{enumerate}
        \item \label{itemone} The map \( D^{\alpha}_{+} F_x \) is homogeneous of degree \(\alpha\).
        \item \label{itemtwo} If \( \alpha > 0 \), then \( F \) is continuous at \( x \).
        \item \label{itemthree} If a \(C^{k}\) function \( f: U \subset \mathbb{R}^n \to  \mathbb{R} \) is of order \(\alpha \leq k\)  at \(x_0\), then \(\mathrm{in}_{x_0}(f)(x) = D^{\alpha}_{+}f_{x_0}(x-x_0)\)
        \item \label{itemfour} If for a given lipeomorphism \( \phi \colon U \to U \), the derivatives \( D^{\alpha}_{+}(F \circ \phi)_x \) and \( D^{\alpha}_{+}F_{\phi(x)} \) exist, then we have
\[
D^{\alpha}_{+} (F \circ \phi)_x = D^{\alpha}_{+} F_{\phi(x)} \circ d_x \phi.
\]

    \end{enumerate}    
\end{proposition}

Before presenting the proof of this statement, let us recall the definition
of the initial part of a function.

\begin{definition}
Let \( x_0 \in U \) with \( f(x_0) = 0 \). If there exists \( \alpha \geq 0 \) such that \( f(x) \) can be written as \( f(x) = f_\alpha(x) + o_\alpha(x) \), where 
\[
\lim_{x \to x_0} \frac{f_\alpha(x)}{\|x - x_0\|^\alpha} \neq 0
\quad \text{and} \quad
\lim_{x \to x_0} \frac{o_\alpha(x)}{\|x - x_0\|^\alpha} = 0,
\]
then \( f_\alpha \) is called the \textbf{initial part} of \( f \) at \( x_0 \), and we denote it by \linebreak \( \mathrm{in}_{x_0}(f) := f_\alpha \). If \( f(x_0) \neq 0 \), we define \( \mathrm{in}_{x_0}(f) \equiv f(x_0) \).
\end{definition}

\begin{proof}

The properties (1), (2) and (3) follow directly from the definition. To establish (4), let \( x \in U \), and choose a vector \( v \in \mathbb{R}^n \setminus \{0\} \) which ensures that the order \( \operatorname{ord}_x F \) is well-defined. Set \( y = \phi(x) \) and \( w = d_x \phi(v) \).

Since the \(\alpha\)-derivatives \( D^{\alpha}_{+}(F \circ \phi)_x \) and \( D^{\alpha}_{+}F_{\phi(x)} \) exist, it suffices to show that their difference vanishes along the sequence \( (t_j)_{j \in \mathbb{N}} \), where \( t_j = \frac{1}{j} \). The definition of the pseudo-Lipschitz derivative implies that
    \[
    \phi(x + t_j v) - \phi(x) = t_j d_x \phi(v) + o(t_j),
    \]
    where \( \lim_{j \to +\infty} \frac{o(t_j)}{t_j} = 0 \). Note that
    \[
    w_j = \frac{\phi(x + t_j v) - \phi(x)}{t_j} \longrightarrow w \quad \text{and} \quad \frac{F(\phi(x + t_j v)) - F(y + t_j w)}{t_j^{\alpha}} \longrightarrow 0.
    \]
    
    Indeed, by the continuity of \( D_{+}^{\alpha} F_y \), we have that
    \[
    \frac{F(y + t_j w_j) - F(y + t_j w)}{t_j^{\alpha}} = \frac{t_j^{\alpha} (D^{\alpha}_{+} F_y(w_j) - D^{\alpha}_{+} F_y(w)) + o_{\alpha}(t_j w) - {o}_{\alpha,j}(t_j w_j)}{t_j^{\alpha}}
    \]
    goes to zero as \( j \to +\infty \), where $\dfrac{{o}_{\alpha,j}(t_j w_j)}{t_j^{\alpha}} \longrightarrow 0$ by a diagonal argument.

\end{proof}

\begin{theorem}\label{inv_order} 
Let \( F \colon U \to \mathbb{R}^p \) and \( G \colon V \to \mathbb{R}^p \), where \( U, V \subset \mathbb{R}^n \) are open, be mappings such that the derivatives \( D_{+}^\alpha F_x \) and \( D_{+}^\beta G_{\phi(x)} \) exist, where \( \phi \colon U \to V \) is a lipeomorphism. If \( F \sim_{x} G \circ \phi \), then
\(
\operatorname{ord}_x F = \operatorname{ord}_{\phi(x)} G.
\)
\end{theorem}

\begin{proof}
 Using the notation established in the Proposition \ref{Prop:Properties}, assume for contradiction, that \( \beta = \operatorname{ord}_{y} G < \operatorname{ord}_x F \). 

By hypothesis, there exists \( c > 0 \) such that
   \begin{eqnarray*}
        \frac{1}{t^{\beta}_j} \|G(y + t_jw)\| & \leq &  \frac{1}{t^{\beta}_j} \|G(y + t_jw) - G(\phi(x + t_jv))\| + \frac{1}{t^{\beta}_j}\|G(\phi(x + t_jv)) \| \\
        &\leq& \frac{1}{t^{\beta}_j} \|G(y + t_jw) - G(\phi(x + t_jv))\| + \frac{c}{t^{\beta}_j}\|F(x + t_jv)\|.
    \end{eqnarray*}

    Notice that the right side of the inequality tends to zero when \( j \to +\infty \), as shown by the proof of Proposition \ref{Prop:Properties}. By the other hand the left side cannot tend to zero since \( \beta = \operatorname{ord}_{y} G \). This is a contradiction! Similarly, using \( (d_x \phi)^{-1} \), we can show that \( \alpha = \operatorname{ord}_x F < \operatorname{ord}_{y} G \) is also impossible. Thus, we conclude that \( \alpha = \beta \).
\end{proof}

%For functions \( f \colon U \subset \mathbb{R}^n \to \mathbb{R} \), the corresponding statement is similar. Moreover, there exists an asymptotic equivalence between the \emph{initial part} of \( f \) and \(g\).

%{\color{red} If $in(f)$ is continuous, then $in(f) = D^{\alpha}f$. Não é verdade para $f(x_0) \neq 0$. A reciproca não é verdade! ver os exemplos \ref{Ex:HalfOrd} \ref{Ex:OrdZero}. Chegamos em $f_{\alpha}(x_0 + tv) + o_{\alpha}(x_0+tv) - \tilde{o}_{\alpha}(tv) = t^{\alpha}D^{\alpha}f_{x_0}(v)$. Não verificamos que a parte inicial é homogênea.}

%\begin{proposition}\label{inv_mult}
%Let \( f \colon (U, 0) \to (\mathbb{R}, 0) \) and \( g \colon (V, 0) \to (\mathbb{R}, 0) \) be germs of functions, where \( U, V \subset \mathbb{R}^n \) are open subsets. Suppose that \( f \) and \( g \) are asymptotically Lipschitz equivalent at \( x \) via the germ of lipeomorphism \( \phi \colon (U, 0) \to (V, 0) \). If \(\mathrm{ord}_x(f) \) and \( \mathrm{ord}_{\phi(x)}(g) \) are defined, then 
%    \[
%    \mathrm{ord}_x(f) = \mathrm{ord}_{\phi(x)}(g)
%    \quad \text{and} \quad 
%    \mathrm{in}_x(f) \sim \mathrm{in}_{\phi(x)}(g) \circ d_x\phi.
%    \]
%\end{proposition}

%The condition \( f(0) = 0 = g(0) \) ensures the invariance of the order, as stated in Theorem \ref{inv_order}.

By imposing the additional requirement of asymptotically Lipschitz equivalent, we obtain the ``asymptotic chain rule", as described below. The proof is similar to the one given in Proposition \ref{Prop:Properties}.

\begin{theorem}
   Let \( F \colon U \to \mathbb{R}^p \) and \( G \colon V \to \mathbb{R}^p \) be mappings that are asymptotically Lipschitz equivalent at \(x\), by a lipemorphism \(\phi\). If the \(\alpha\)-derivatives of \(F\) at \(x\) and of \(G\) at \(\phi(x)\) exist, then
   \[
   D^{\alpha}_{+}F_{x} \sim D^{\alpha}_{+}G_{\phi(x)} \circ d_{x}\phi.
   \]
\end{theorem}

%\begin{proof}
%  Using the notation established in the previous discussion, we obtain the following inequality:
%    \begin{eqnarray*}
%        \frac{1}{t^{\alpha}_j} \|F(x + t_jv) - F(x) \| & \leq &  \frac{c_2}{t^{\alpha}_j}\|G(\phi(x + t_jv)) - G(\phi(x)) \| \\
%        &\leq& \frac{c_2}{t^{\alpha}_j} \|G(\phi(x + t_jv)) -G(y + t_jw) \| + \frac{c_2}{t^{\alpha}_j} \|G(y + t_jw) - G(y)\|.
%    \end{eqnarray*}
%Thus, we obtain the estimate  \( \|D^{\alpha}_{+}F_x(v)\| \leq c_2 \|D^{\alpha}_{+}G_y(w)\| \) ({\color{red} 
%WHY YOU DO NOT NEED THE FACT THAT G IS CONTINUOUS HERE BUT YOU NEEDED IT WHEN YOU STATED THE PROPERTIES? I BELIEVE THE SAME PROBLEM HAPPENS HERE AND WE MUST ASK THIS} Precisamos apenas garantir a existência das derivadas, as quais são contínuas por definição).

%On the other hand, considering \( (d_x \phi)^{-1} \), we obtain the reverse estimate \( \|D^{\alpha}_{+}G_y(w)\| \leq c_3 \|D^{\alpha}_{+}F_x(v)\| \) for some constant \( c_3 > 0 \).
%\end{proof}

If a map \(F\) is differentiable at a point \(x \in \mathbb{R}^n\), then we can consider the standard derivative \(D_xF\).  Under this assumption, we deduce some invariants from the preceding discussion.

\begin{corollary}
Let \( F \colon U \to \mathbb{R}^p \) and \( G \colon V \to \mathbb{R}^p \) be differentiable mappings, where \( U, V \subset \mathbb{R}^n \) are open subsets. Suppose there exists a lipeomorphism \( \phi \colon U \to V \) such that \( F \approx_{x} G \). Then,
\begin{enumerate}
 \item \label{coritemone} \(d_x \phi(\ker(D_xF)) = \ker(D_{\phi(x)}G)\).
 \item \label{coritemtwo} \(\operatorname{rank}(D_xF) = \operatorname{rank}(D_{\phi(x)}G).
\)
 \item \label{coritemthree} \(\phi(\operatorname{Sing}(F)) = \operatorname{Sing}(G).\)
 \end{enumerate}
\end{corollary}
\begin{proof}
    Consider \( w \in d_x \phi(\ker(D_x F)) \). Then, there exists \( v \in \ker(D_x F) \) with \( w = d_x \phi(v) \), and a constant \(c_1 > 0\) such that
    \[
    ||D_{\phi(x)} G(w)|| = ||D_{\phi(x)} G(d_x \phi(v))|| \leq c_1||D_x F(v)|| = 0.
    \]

    Conversely, let \( w \in \ker(D_{\phi(x)} G) \subset \mathbb{R}^n \). Since \( d_x \phi \colon \mathbb{R}^n \rightarrow \mathbb{R}^n \) is a homeomorphism, there exists \( v \in \mathbb{R}^n \) such that \( w = d_x \phi(v) \). Thus,
    \[
    0 = ||D_{\phi(x)} G(w)|| = ||D_{\phi(x)} G(d_x \phi(v))|| \geq c_2 ||D_x F(v)||,
    \]
for some constant \(c_2 > 0\).
    Therefore, \( d_x \phi(\ker(D_x F)) = \ker(D_{\phi(x)} G) \). In particular, this implies \( \dim \ker(D_x F) = \dim \ker(D_{\phi(x)} G) \), and consequently, \( \operatorname{rank}(D_x F) = \operatorname{rank}(D_{\phi(x)} G) \).

    Furthermore, note that \( x \in \mathbb{R}^n \) is a singular point of \( F \) if and only if \( p > \operatorname{rank}(D_x F) = \operatorname{rank}(D_{\phi(x)} G) \). We conclude that \( x \in \operatorname{Sing}(F) \) if and only if \( \phi(x) \in \operatorname{Sing}(G) \).
\end{proof}

%\begin{corollary} Let \( F \colon U \to Z \) and \( G \colon V \to W \) be two differentiable mappings, where \( U, V \subset \mathbb{R}^n \) (resp. \( Z, W \subset \mathbb{R}^k \)) are open subsets. If \( F \) is \(\mathcal{A}\)-bi-Lipschitz equivalent to \( G \), then their singular sets are Lipschitz invariants. That is, there exists a lipeomorphism \( \phi \colon U \to V \) such that \( \phi(\operatorname{Sing}(F)) = \operatorname{Sing}(G) \). \end{corollary}

\section{Invariance of the Milnor Number under Lipschitz equivalences}\label{section:invarianceMilnor}

%We now focus on germs of \(C^{\infty}\) functions \(f\colon (\R^n,0)\to (\R,0)\). 

For simplicity, we will use the notation \({\rm ord}(f) = {\rm ord}_0(f)\) and \({\rm in}(f) = {\rm in}_0(f)\). In what follows, we examine the invariance of the Milnor number \(\mu(f)\) under the bi-Lipschitz equivalences previously discussed.

\begin{lemma}\label{Lem:homotopy}
   Let \(g\colon (\C^n,0)\to (\C,0)\) be a germ of analytic function such that \({\rm in}(g) = g_m\) has an isolated singularity at the origin. For each \( t \in [0,1] \), set  \(h_t(x) = t^{-m} g(tx)\). Then, there exists a homotopy between  \(\dfrac{\nabla g_m}{|\nabla g_m |}\) and \(\dfrac{\nabla h_{t_0}}{|\nabla h_{t_0} |}\), for some \(0 < t_0 \leq 1\).
\end{lemma}
\begin{proof}
Consider the map
\[
H: \mathbb{S}^{2n-1}_{\varepsilon} \times [0,1] \longrightarrow \mathbb{S}^{2n-1}_{\varepsilon},
\quad (x,t) \longmapsto \frac{\nabla h_t(x)}{|\nabla h_t(x)|},
\]
where \( 0 < \varepsilon \ll 1 \) is a Milnor radius for \( g_m \).

If we express \(g\) as  
\[
g = g_m + \sum_{i\geq1} g_{m+i},
\]
where each \(g_{m+i}\) is homogeneous of degree \(m+i\), then \[
\nabla h_t(x) = \nabla g_m(x) + \sum_{i\geq1} t^i \nabla g_{m+i}(x).
\] 
 
We claim that there exists \(0< t_0 \leq 1\) such that \(\nabla h_t(x) \neq 0\) for all \(x \in \mathbb{S}^{2n-1}_{\varepsilon}\) and \(t \leq t_0\). Suppose otherwise. Then, there exist sequences \(t_j \to 0\) and \(x_j \to x_0 \in \mathbb{S}^{2n-1}_{\varepsilon}\) such that \(\nabla h_{t_j}(x_j) = 0\) for each \(j \in \mathbb{N}\). This implies \(\nabla g_m(x_0) = 0\), which contradicts the fact that \(g_m\) has an isolated singularity at the origin.  

Consequently, \( H \) provides the desired homotopy.  
\end{proof}

\begin{proposition}
\label{milnor_ord}
    Let \(f\colon (\R^n,0)\to (\R,0)\) be a germ of \(C^{\infty}\) function. If \({\rm in}(f)\) has an algebraically isolated singularity at the origin, then the Milnor number of \(f\) satisfies \(\mu(f) = ({\rm ord}(f) - 1)^n\).
\end{proposition}
\begin{proof}
Let \( f_m = {\rm in}(f) \), where \( f_m \) is a homogeneous function of degree \( m = {\rm ord}(f) \). By considering the complexification of \( f_m \), Theorem 1 in \cite{MilnorO:1970} ensures that \(\mu(f_m) = ({\rm ord}(f)-1)^n\).

Thus, \( f_m \) is \( k \)-\(\mathcal{R}\)-\(C^{\infty}\) determined, for \( k = ({\rm ord}(f)-1)^n + 1\) (see \cite{Ruas:1986}).

Now, consider the \( k \)-th jet of \( f \), denoted by \( g = j^k f \). Since \( k \geq {\rm ord}(f) \), we have that \( {\rm in}(g) = {\rm in}(f) \), and we denote this initial part as \( g_m \).

Complexifying \( g_m \), Milnor \cite{Milnor:1968} guarantees that
\[
\mu(g_m) = \deg\left(\frac{\nabla g_m}{|\nabla g_m|}\right).
\]

By Lemma \ref{Lem:homotopy}, we have \( \mu(g_m) = \mu(h_{t_0}) \) for some \( t_0 > 0 \). Since \( \varphi(x) = t_0 x \) is a diffeomorphism, it follows that
\[
\mu(g) = \mu(g \circ \varphi) = \mu(h_{t_0}) = \mu(g_m) = \mu(f_m).
\]

Finally, \( g \) is also \( k \)-\(\mathcal{R}\)-\(C^{\infty}\) determined, therefore \( f \) and \( g \) are \(\mathcal{R}\)-\(C^{\infty}\) equivalent. We conclude that,
\[
\mu(f) = \mu(g) = ({\rm ord}(f)-1)^n.
\]
\end{proof}

Note that if \(f\) and \(g\) are \(\mathcal{R}\)-bi-Lipschitz equivalent, then they are, in particular, \(\mathcal{R}\)-asymptotically Lipschitz equivalent at the origin. By combining Theorem \ref{inv_order} with Proposition \ref{milnor_ord}, we deduce the following. 

\begin{theorem} \label{Cor:Invariance}
    Let \(f, g\colon (\R^n,0)\to (\R,0)\) be germs of \(C^{\infty}\) functions. Suppose that \({\rm in}(f)\) and \({\rm in}(g)\) have algebraically isolated singularities at the origin. If \(f\) and \(g\) are \(\mathcal{R}\)-asymptotically Lipschitz equivalent  at the origin, then \(\mu(f) = \mu(g)\).
\end{theorem}

We illustrate with the following example that, in general, \( \phi \colon \mathbb{R}^n \to \mathbb{R}^n \) cannot be considered a bi-\(\alpha\)-Hölder map for any \( \alpha \in (0,1) \), \textit{i.e.}, a bijective map satisfying, for some constant \( c > 0 \), 
\[
\frac{1}{c} \|x - y\|^{\alpha} \leq \|\phi(x) - \phi(y)\| \leq c \|x - y\|^{\alpha}, \quad \text{for all } x, y \in \mathbb{R}^n.
\]

\begin{example}\label{Ex:AlphaHolder}
Let \( f, g \colon \mathbb{R}^2 \to \mathbb{R} \) be the functions defined by  
\[
f(x,y) = x^{2k+1} + y^{2k+1}, \quad g(x,y) = x^{2k+3} + y^{2k+3},
\]  
for a fixed \(k \geq 1\).

Consider the map \( \phi \colon \mathbb{R}^2 \to \mathbb{R}^2 \) given by \(\phi(x,y) = \left(x^{\frac{2k+1}{2k+3}}, y^{\frac{2k+1}{2k+3}} \right).\) This is a bi-\(\alpha\)-Hölder map, \(\alpha = \frac{2k+1}{2k+3}\) (Example 1.1.3 in \cite{Fiorenza:2016}), satisfying \( f = g \circ \phi \). However, the Milnor numbers are \( \mu(f) = 4k^2 \) and \( \mu(g) = (2k+2)^2 \).     
\end{example}

In a more general setting, families of functions that vary in a controlled manner under 
\(\mathcal{R}\)-bi-Lipschitz equivalence should also exhibit invariance of the Milnor number. 

\begin{corollary}
Let $\{f_t\}_{t\in [0,1]}$ be a family of $C^{\infty }$ functions. Suppose that $\{f_t\}$ is $\mathcal{R}$-bi-Lipschitz trivial and $\mu({\rm in}(f_t))<+\infty$ for all $t\in [0,1]$. Then $\mu(f_t)=\mu(f_0)$ for all $t\in [0,1]$.
\end{corollary}

%Consequently, we obtain the following result.

%\begin{corollary}
%    Let \(f, g\colon (\R^n,0)\to (\R,0)\) be germs of \(C^{\infty}\) functions. Suppose that \( {\rm in}(f), {\rm in}(g)\) have algebraically isolated singularities at the origin. If \(f\) and \(g\) are \(\mathcal{R}\)-bi-Lipschitz equivalent, then \(\mu(f) = \mu(g)\).
%\end{corollary}

\begin{example}[\cite{NguyenRT:2020}]\label{exam:Lip_trivial_milnor_num_different}
For each \( t \in [0,1] \), consider the function \( f_t \colon \mathbb{R}^2 \to \mathbb{R} \) given by  
\[
f_t(x,y) = x^4 + y^4 + 2t(x^2y^2 + y^6).
\]  

The family \(\{f_t\}\) is \(\mathcal{R}\)-bi-Lipschitz trivial, and for all \( t \in [0,1) \), the initial part \({\rm in}(f_t)\) has a finite Milnor number. However, note that at \( t = 1 \), the initial part  
\(
{\rm in}(f_1) = (x^2 + y^2)^2
\)  
does not have an algebraically isolated singularity at the origin. Moreover, the Milnor numbers satisfy \(\mu(f_0) = 9\) and \(\mu(f_1) = 13\). 
\end{example}

Note that the Milnor number is not preserved under $\mathcal{R}$-bi-Lipschitz equivalence, even within a family $\{g_t\}$ of homogeneous polynomials that is $\mathcal{R}$-bi-Lipschitz trivial.

\begin{example}
Let \( g_t = \operatorname{in}(f_t) \), where the family \( \{f_t\} \) is defined in Example~\ref{exam:Lip_trivial_milnor_num_different}. Since \( \{f_t\} \) is \( \mathcal{R} \)-bi-Lipschitz trivial, it follows that the family \( \{g_t\} \) is also \( \mathcal{R} \)-bi-Lipschitz trivial. Nevertheless, the Milnor number is not constant: we have \( \mu(g_1) = +\infty \), while \( \mu(g_t) = 9 \) for all \( t \neq 1 \).
\end{example}

\section{Invariance of the Milnor Number under $C^k$ equivalences}\label{section:diffeo}

We recall that two \( C^{\infty} \) function-germs \( f, g \colon (\mathbb{R}^n,0) \to (\mathbb{R},0) \) are said to be \(\mathcal{R}\)-\(C^k\) equivalent, for \( k \in \mathbb{N} \cup \{\infty\} \), if there exists a \( C^k \) diffeomorphism \( \phi \colon (\mathbb{R}^n,0) \to (\mathbb{R}^n,0) \) such that \(f = g \circ \phi.
\) In particular, if \( k \geq 1 \) and \( f, g \) are \(\mathcal{R}\)-\(C^k\) equivalent, then they are \(\mathcal{R}\)-bi-Lipschitz equivalent. 

Even though \(\mathcal{R}\)-\(C^k\) equivalence for \( k \geq 1 \) is a stronger condition, we can show that it does not, \textit{a priori}, require both initial parts to have an algebraically isolated singularity at the origin. Indeed, if \(\phi\colon (\mathbb{R}^n,0) \to (\mathbb{R}^n,0)\) is a \(C^k\) diffeomorphism with \( k \geq 1 \) such that \( f = g \circ \phi \), then its derivative \( D_0\phi \colon \mathbb{R}^n \to \mathbb{R}^n \) is a linear isomorphism satisfying  
\[
{\rm in}(f) = {\rm in}(g) \circ D_0\phi.
\] 

In other words, \({\rm in}(f)\) and \({\rm in}(g)\) are \(\mathcal{R}\)-\(C^{\infty}\) equivalent. As a consequence, if we assume, for instance, that \(\mu({\rm in}(f)) < +\infty\), it follows necessarily that \(\mu({\rm in}(g)) < +\infty\). With this in mind, we proceed to state the following.

\begin{theorem}\label{Teo:CkEqui}
    Let \( f, g \colon (\mathbb{R}^n,0) \to (\mathbb{R},0) \) be germs of \( C^{\infty} \) functions. Suppose that  \({\rm in}(f)\) has an algebraically isolated singularity at the origin. If \( f \) and \( g \) are \(\mathcal{R}\)-\(C^1\) equivalent, then \(\mu(f) = \mu(g)\).
\end{theorem}

This result naturally extends to trivial \(\mathcal{R}\)-\(C^k\) families of functions, as shown in the next corollary.

\begin{corollary}
Let \(\{f_t\}_{t\in [0,1]}\) be a family of \(C^{\infty}\) functions. Suppose that \(\{f_t\}\) is \(\mathcal{R}\)-\(C^1\) trivial and that \({\rm in}(f_0)\) has an algebraically isolated singularity at the origin. Then, \(\mu(f_t) = \mu(f_0)\), for all \(t \in [0,1]\).
\end{corollary}

\begin{example}
For each \( t \in [0,1] \), consider the function \( f_t \colon \mathbb{R}^2 \to \mathbb{R} \) given by  
\[
f_t(x,y) = x^4 + 2x^2y^2 + y^4 + t y^{k+5}.
\]  
The family \( \{f_t\} \) is \(\mathcal{R}\)-\(C^{k+1}\) trivial (see \cite{Saia:2008}). It is shown in \cite{Kuiper:1972} that this family is not \(\mathcal{R}\)-\(C^{k+2}\) trivial.  

Note that for \(0 < t \leq 1\), we have \(\mu(f_t) = 2(k+5) + 1\), whereas \(\mu(f_0) = +\infty\). Moreover, the initial part \({\rm in}(f_t) = (x^2 + y^2)^2\) does not have an algebraically isolated singularity at the origin.
\end{example}

As shown in Example \ref{Ex:NotC1}, the real Milnor number is not a \(C^1\) invariant. This means that there exist germs of \(C^\infty\) functions \( f, g \colon (\mathbb{R}^n,0) \to (\mathbb{R},0) \) and a \(C^1\) diffeomorphism \(\phi \colon (\mathbb{R}^n, V(f), 0) \to (\mathbb{R}^n, V(g), 0)\) such that \(\mu(f) \neq \mu(g)\). This example can be extended to the case of \(C^k\) equivalence, as follows.

\begin{example}
Let \( f, g \colon \mathbb{R}^2 \to \mathbb{R} \) be the functions defined by  
\[
f(x,y) = x^{3k+1} - y^3, \quad g(x,y) = y.
\]  
Consider the map \( \phi \colon \mathbb{R}^2 \to \mathbb{R}^2 \) given by \(\phi(x,y) = \left(x, x^{k+\frac{1}{3}} - y \right).\) 
This is a \(C^k\) diffeomorphism satisfying \( \phi(V(f)) = V(g) \). However, the Milnor numbers are \( \mu(f) = 6k \) and \( \mu(g) = 0 \). We emphasize that \( \phi \) is not a \( C^{k+1} \) diffeomorphism.
\end{example}

\begin{theorem} \label{Teo:Cinfinito}
Let \( f, g \colon (\mathbb{R}^n,0) \to (\mathbb{R},0) \) be irreducible germs of real analytic functions. If there exists a \( C^{\infty} \) diffeomorphism  
\(\phi \colon (V(f), 0) \to (V(g), 0)\), then  \( \mu(f) = \mu(g) \).
\end{theorem}
\begin{proof}
By Proposition 1.1 in \cite{Ephraim:1973}, there exists a real analytic diffeomorphism \( \Phi \colon V(f) \to V(g) \). By complexifying this analytic diffeomorphism, it can be extended, using the Lifting Lemma (see Remark 1.30.1 in \cite{Greuel-Lossen-Shustin:2007}), to an analytic diffeomorphism \( \Phi_{\mathbb{C}} \colon (\mathbb{C}^n,0) \to (\mathbb{C}^n,0) \) such that \( \Phi_{\mathbb{C}}(V(f_{\mathbb{C}})) = V(g_{\mathbb{C}}) \). In particular, since \( \Phi_{\mathbb{C}} \) is a homeomorphism, Lê's result \cite{Le:1973} guarantees that \( \mu(f_{\mathbb{C}}) = \mu(g_{\mathbb{C}}) \). The statement follows from Proposition \ref{Prop:Complexifing}.
\end{proof}

\noindent{\bf Acknowledgements}. 

The first author was partially supported by the FUNCAP post-doctorate fellowship number FPD-0213-00092.01.00/23

\end{document}